\documentclass[a4, 12pt]{amsart}
\usepackage{amssymb}
\usepackage{amstext}
\usepackage{amsmath}
\usepackage{amscd}
\usepackage{latexsym}
\usepackage{amsfonts}

\theoremstyle{plain}
\newtheorem{thm}{Theorem}[section]
\newtheorem*{thm*}{Theorem}
\newtheorem*{cor*}{Corollary}

\newtheorem{prop}[thm]{Proposition}
\newtheorem{lem}[thm]{Lemma}
\newtheorem{cor}[thm]{Corollary}
\newtheorem{claim}{Claim}
\newtheorem*{claim*}{Claim}

\theoremstyle{definition}
\newtheorem{defn}[thm]{Definition}
\newtheorem{ex}[thm]{Example}

\newtheorem{fact}[thm]{Fact}

\theoremstyle{remark}

\numberwithin{equation}{thm}
\def\Hom{\mathrm{Hom}}

\def\Proj{\operatorname{Proj}}

\def\Coker{\mathrm{Coker}}

\def\Im{\mathrm{Im}}

\def\a{\mathfrak a}
\def\b{\mathfrak b}

\def\e{\mathrm{e}}
\def\g{\mathrm{g}}
\def\m{\mathfrak m}
\def\n{\mathfrak n}

\def\p{\mathfrak p}
\def\q{\mathfrak q}

\def\Z{\Bbb Z}

\def\H{\mathrm{H}}

\newcommand{\rme}{\mathrm{e}}

\newcommand{\rmE}{\mathrm{E}}

\newcommand{\rmH}{\mathrm{H}}

\newcommand{\rmR}{\mathrm{R}}

\newcommand{\rmT}{\mathrm{T}}
\newcommand{\rmU}{\mathrm{U}}

\newcommand{\calF}{\mathcal{F}}

\newcommand{\fkm}{\mathfrak{m}}
\newcommand{\fkn}{\mathfrak{n}}

\newcommand{\fkq}{\mathfrak{q}}

\def\depth{\mathrm{depth}}

\def\Ass{\mathrm{Ass}}
\def\Assh{\mathrm{Assh}}

\def\Spec{\mathrm{Spec}}

\def\hdeg{\operatorname{hdeg}}

\def\gr{\mbox{\rm gr}}

\begin{document}

\setlength{\baselineskip}{14pt}
\title{The second Hilbert coefficients and\\ the homological torsions of parameters}
\pagestyle{plain}
\author{Shiro Goto}
\address{Department of Mathematics, School of Science and Technology, Meiji University, 1-1-1 Higashi-mita, Tama-ku, Kawasaki 214-8571, Japan}
\email{goto@math.meiji.ac.jp}
\author{Kazuho Ozeki}
\address{Department of Mathematical Science, Faculty of Science, Yamaguchi University, 1677-1 Yoshida, Yamaguchi 753-8512, Japan}
\email{ozeki@yamaguchi-u.ac.jp}

\thanks{
\endgraf
{AMS 2010 {\em Mathematics Subject Classification:}
13D40, 13H15,  13H10.}}

\maketitle


{\footnotesize
}

\section{Introduction}

The purpose of our paper is to study the second Hilbert coefficients of parameters in terms of the homological degrees and torsions of modules.

To state the problems and the results of our paper, first of all, let us fix some of our notation. Let $A$ be a Noetherian local ring with  maximal ideal $\fkm$ and $d = \dim A > 0$. Let $M$ be a finitely generated $A$-module with $s = \dim_AM$. For simplicity, throughout this paper, we assume that $A$ is $\m$--adically complete and the residue class field $A/\fkm$ of $A$  is infinite. 

For each $j \in \Z$ we set
$$
M_j  
= \Hom_A(\rmH_\fkm^j(M), E),
$$
where $E=\rmE_A(A/\fkm)$ denotes the injective envelope of $A/\fkm$ and $\H_{\m}^j(M)$ the $j$th local cohomology module of $M$ with respect to the maximal ideal $\m$.
Then $M_j$ is a finitely generated $A$-module with $\dim_A M_j \le j$ for all $j \in \Z$ ([GO2, Fact 2.1]).  Let $I$ be a fixed $\fkm$-primary ideal in $A$ and let $\ell_A(N)$ denote, for an $A$-module $N$, the length of $N$.
Then there exist integers $\{\e_I^i(M)\}_{0 \leq i \leq s}$ such that
$$\ell_A(M/I^{n+1}M)=\e_I^0(M)\binom{n+s}{s}-\e_I^1(M)\binom{n+s-1}{s-1}+\cdots+(-1)^s\e_I^s(M)$$
for all $n \gg 0$. We call $\e_I^i(M)$ the $i$-th Hilbert coefficient of $M$ with respect to $I$ and especially call the leading coefficient $\e_I^0(M)~(>0)$ the multiplicity of $M$ with respect to $I$.

The homological degree $\hdeg_I(M)$ of $M$ with respect to $I$ is inductively defined in the following way, according to the dimension $s = \dim_A M$ of $M$.

\begin{defn}\label{defn1}$($\cite{V2}$)$ For each finitely generated $A$-module $M$ with $s = \dim_A M$, we set 
$$
\hdeg_I(M) = \left\{
\begin{array}{lc}
\ell_A(M) & \mbox{if $s \le 0$},\\
\vspace{1mm}\\
\rme_I^0(M) + \sum_{j=0}^{s-1} \binom{s-1}{j}\hdeg_I(M_j) & \mbox{if $s > 0$}
\end{array}
\right.
$$
and call it the homological degree of $M$ with respect to $I$.
\end{defn}

Then the homological torsions of modules is defined as follows.

\begin{defn}\label{defn2}
Let $M$ be a finitely generated $A$--module with $s= \dim_A M \ge 2$. We set $$\rmT_I^i(M) = \sum_{j=1}^{s-i}\binom{s-i-1}{j-1} \hdeg_I(M_j)$$
for each $1 \leq i \leq s-1$ and call them the homological torsions of $M$ with respect to $I$.
\end{defn}

Notice that the homological degrees $\hdeg_I(M)$ and torsions $\rmT_I^i(M)$ of $M$ with respect to $I$ depend only on the integral closure of $I$.

In \cite{GhGHOPV2}, it was proved that for parameter ideals $Q$ for $M$, a lower bound
$$ \e_Q^1(M) \geq - \rmT_Q^1(M) $$
of the first Hilbert coefficient $\e_Q^1(M)$ in terms of the homological torsion $\rmT_Q^1(M)$.
Recently, in \cite[Theorem 1.4]{GO2}, the authors showed that the equality $\e_Q^1(M)=-\rmT_Q^1(M)$ holds true if and only if the equality $\chi_1(Q;M)=\hdeg_Q(M)-\e_Q^0(M)$ holds true for parameter ideals $Q$ for an unmixed module $M$, where $\chi_1(Q;M)=\ell_A(M/QM)-\e_Q^0(M)$ denotes the first Euler characteristic of $M$ relative to $Q$.
Here we notice that the inequality $\e_Q^1(M) \leq 0$ holds true for every parameter ideals $Q$ of $M$ (\cite[Theorem 3.5]{MSV}) and that $M$ is a Cohen-Macaulay $A$-module once $\e_Q^1(M)=0$ for some parameter ideal $Q$, provided $M$ is unmixed (see \cite{GhGHOPV1, GhGHOPV2}). 
Thus the behavior of the first Hilbert coefficients $\e_Q^1(M)$ for parameter ideals $Q$ for $M$ are rather satisfactory understood.

The purpose of this paper is to study the natural question of how about the second Hilbert coefficients $\e_Q^2(M)$ of $M$ with respect to $Q$.
For the estimation of $\e_Q^2(M)$ the key is the inequalities
$$ -\ell_A(\H_{\m}^1(M)) \leq \e_Q^2(M) \leq 0 $$
of parameters $Q$ in the case where $d=\dim_AM=2$ and $\depth_A M >0$ (Proposition 3.1).
We will also show that $\e_Q^2(M)=0$ if and only if the ideal $Q$ is generated by a system $a_1,a_2$ of parameters which forms a $d$-sequence on $M$.
Then the first main result of this paper answers the question and is stated as follows. Recall that $M$ is said to be unmixed, if $\dim A/\p=\dim_A M$ for all $\p \in \Ass_AM$ (since $A$ is assumed to be $\m$-adically complete).

\begin{thm}\label{thm1}
Let $M$ be a finitely generated $A$-module with $d =\dim_AM \ge 3$.
Then we have the following:
\begin{itemize}
\item[(1)] we have
$$\rme_Q^2(M) \leq \rmT_Q^2(M)$$
for every parameter ideals $Q$ of $A$.
\item[(2)] We have $$-\sum_{j=2}^{d-1} \binom{d-3}{j-2} \hdeg_Q(M_j) \le \rme_Q^2(M)$$ for every parameter ideals $Q$ of $A$, if $M$ is unmixed.
\end{itemize}
\end{thm}

Therefore we get the finiteness of the set
$$\Lambda_Q^2(M)=\{\e_{\q}^2(M) \ | \ \mbox{$\fkq$ is a parameter ideal of $M$ such that $\overline{\fkq} = \overline{Q}$}  \}$$
for parameters of the unmixed module $M$ with $\dim_A M \geq 2$ (Corollary 3.6), where, for an ideal $J$ in $A$, $\overline{J}$ denotes the integral closure of $J$.
We note here that, unless $M$ is unmixed, the inequality (2) in Theorem 1.3 does not hold true in general (Example 3.7).

Thus the second Hilbert coefficients $\e_Q^2(M)$ bounded by above in terms of the homological torsions $\rmT_Q^2(M)$.
It seems now natural to ask what happens on the parameters $Q$ of $M$ once the equality $ \e_Q^2(M) = \rmT_Q^2(M) $ is attained.
Let $\g_s(Q;M)=\ell_A(M/QM)-\e_Q^0(M)+\e_Q^1(M)$ denotes the sectional genera of $M$ with respect to $Q$.
We notice here that, in \cite{GO1}, we explored the relationship between the sectional genera and the homological degrees of parameters, and gave a criterion for the equality $\g_s(Q;M)=\hdeg_Q(M)-\e_Q^0(M)-\rmT_Q^1(M)$.

Then the second main result of this paper answers the question and is stated as follows (Theorem 4.1), where the sequence $a_1,a_2,\ldots,a_d$ is said to be  a $d$-sequence on $M$, if the equality $$[(a_1,a_2,\ldots,a_{i-1})M:_M a_ia_j]=[(a_1,a_2,\ldots,a_{i-1})M:_M a_j]$$ holds true for all $1 \leq i \leq j \leq d$ (\cite{H}).

\begin{thm}\label{thm2}
Let $M$ be a finitely generated $A$-module with $d = \dim_AM\geq 3$ and suppose that $M$ is unmixed. 
Let $Q$ be a parameter ideal of $A$. 
Then the following conditions are equivalent:
\begin{itemize}
\item[$(1)$] $\g_s(Q;M)=\hdeg_Q(M)-\e_Q^0(M)-\rmT_Q^1(M)$,
\item[$(2)$] $\e_Q^2(M)=\rmT_Q^2(M)$.
\end{itemize}
When this is the case, we have the following$\mathrm{:}$
\begin{itemize}
\item[$(\mathrm{i})$] $(-1)^i\e_Q^i(M)=\rmT_Q^i(M)$ for $3 \leq i\leq d-1$ and $\e_Q^d(M)=0$,
\item[$(\mathrm{ii})$] $\ell_A(M/Q^{n+1}M)=\sum_{i=0}^d(-1)^i\e_{Q}^i(M)\binom{n+d-i}{d-i}$ for all $n \geq 0$,
\item[$(\mathrm{iii})$] there exist elements $a_1,a_2,\ldots,a_d \in A$ such that $Q=(a_1,a_2,\ldots,a_d)$ and $a_1,a_2,\ldots,a_d$ forms a $d$-sequence on $M$, and
\item[$(\mathrm{iv})$] $Q\H_{\m}^i(M)=(0)$ for all $1 \leq i \leq d-3$.
\end{itemize}
\end{thm}

We now briefly explain how this paper is organized.
In Section 2 we will summarize,  for the later use in this paper, some auxiliary results on the homological degrees and torsions.
We shall prove Theorem 1.3 in Section 3 (Theorem 3.3).
Theorem 1.4 will be proven in Section 4 (Theorem 4.1).
Unless $M$ is unmixed, the implication $(2) \Rightarrow (1)$ in Theorem 1.4 does not hold true in general. 
We will show in Section 4 an example of  parameter ideals $Q$ in a three-dimensional mixed local ring $A$ such that $\e_Q^1(A)=-\rmT_Q^1(A)$ but $\chi_1(Q;A) < \hdeg_Q(A)-\e_Q^0(A)$.

In what follows, unless otherwise specified, let $A$ be a Noetherian local ring with maximal ideal $\m$ and $d = \dim A>0$.
Let $M$ be a finitely generated $A$-module with $s=\dim_AM$.
We throughout assume that $A$ is $\m$--adically complete and the field $A/\m$ is infinite.
For each $\m$-primary ideal $I$ in $A$ we set
$$ R=\rmR(I)=A[It], \ \ R'=\rmR'(I)=A[It,t^{-1}], \ \ \mbox{and} \ \ \gr_I(A)=\rmR'(I)/t^{-1}\rmR'(I),$$
where $t$ is an indeterminate over $A$.


\section{Preliminaries}

In this section we summarize some basic properties of homological degrees and torsions of modules, which we need throughout this paper.
See \cite{GO2} for the detailed proofs.

For each $j \in \Z$ we set
$$
M_j  
= \Hom_A(\rmH_\fkm^j(M), E),
$$
where $E=\rmE_A(A/\fkm)$ denotes the injective envelope of $A/\fkm$ and $\H_{\m}^j(M)$ the $j$th local cohomology module of $M$ with respect to $\m$.
Then, for each $j \in \Z$, $M_j$ is a finitely generated $A$-module with $\dim_A M_j \le j$, where $\dim_A (0) = -\infty$ (\cite[Fact 2.1]{GO2}).

We recall the definition of homological degrees.

\begin{defn}\label{hdeg}$($\cite{V2}$)$
For each finitely generated $A$-module $M$ with $s = \dim_A M$ and for each $\m$-primary ideal $I$ of $A$, we set 
$$
\hdeg_I(M) = \left\{
\begin{array}{lc}
\ell_A(M) & \mbox{if $s \le 0$},\\
\vspace{1mm}\\
\rme_I^0(M) + \sum_{j=0}^{s-1} \binom{s-1}{j}\hdeg_I(M_j) & \mbox{if $s > 0$}
\end{array}
\right.
$$
and call it the homological degree of $M$ with respect to $I$.
\end{defn}

Let us summarize some basic properties of $\hdeg_I(M)$.

\begin{fact}\label{fact2}
Let $M$ and $M'$ are finitely generated $A$-modules. Let $I$ be an $\m$-primary ideal in $A$.
Then $0 \le \hdeg_I(M) \in \Z$. We furthermore have the following:
\begin{itemize}
\item[(1)] $\hdeg_I(M) = 0$ if and only if $M=(0)$.
\item[(2)] If $M \cong M'$, then $\hdeg_I(M) = \hdeg_I(M')$.
\item[(3)] $\hdeg_I(M)$ depends only on the integral closure of $I$.
\item[(4)] If $M$ is a generalized Cohen-Macaulay $A$-module, then
$$\hdeg_I(M) - \rme_I^0(M) = \mathbb{I}(M)$$
and
$$ \ell_A(M/QM)-\e_Q^0(M) \leq \mathbb{I}(M)$$
for all parameter ideals $Q$ for $M$ $($\cite{STC}$)$, where $\mathbb{I}(M) = \sum_{j=0}^{s-1}\binom{s-1}{j} \ell_A(\H_{\m}^j(M))$ denotes the St\"{u}ckrad-Vogel invariant of $M$.
\end{itemize}
\end{fact}

The following result plays a key role in the analysis of homological degree.

\begin{lem}\label{xyz}$(${\cite[Proposition 3.18]{V2}}$)$
Let $0 \to X \to Y \to Z \to 0$ be an exact sequence of finitely generated $A$-modules. Then the following assertions hold true:
\begin{itemize}
\item[(1)] If $\ell_A(Z) < \infty$, then $\hdeg_I(Y) \le \hdeg_I(X) + \hdeg_I(Z)$.
\item[(2)] If $\ell_A(X) < \infty$, then $\hdeg_I(Y) = \hdeg_I(X) + \hdeg_I(Z)$.
\end{itemize}
\end{lem}

\begin{proof}
See \cite[Lemma 2.4]{GO2}.
\end{proof}

Let $R=\rmR(I)=A[It] \subseteq A[t]$ be the Rees algebra of $I$ (here $t$ denotes an indeterminate over $A$) and let 
$
f : I \to R, ~a \mapsto at
$
be the identification of $I$ with $R_1 = It$. Set $$\Proj R = \{\p \mid \p \text{ is a graded prime ideal of} ~R ~\text{such that}~\p \not\supseteq R_+\}.$$
We then  have the following.

\begin{lem}\label{superficial1}$(${\cite[Theorem 2.13]{V1}}$)$
Let $M$ be a finitely generated $A$-module. Then there exists a finite subset $\calF \subseteq \Proj R$ such that 
\begin{enumerate}
\item[$(1)$] every $a \in I \setminus \bigcup_{\p \in \calF}[f^{-1}(\p) + \fkm I]$ is superficial for $M$ with respect to $I$ and 
\item[$(2)$] $\hdeg_I(M/aM) \le \hdeg_I(M)$ for each $a \in I \setminus \bigcup_{\p \in \calF}[f^{-1}(\p) + \fkm I]$.
\end{enumerate}
\end{lem}

\begin{proof}
See \cite[Lemma 2.6]{GO2}.
\end{proof}

We recall the definition of homological torsions.

\begin{defn}\label{torsion}
Let $M$ be a finitely generated $A$--module with $s = \dim_A M \ge 2$. We set $$\rmT_I^i(M) = \sum_{j=1}^{s-i}\binom{s-i-1}{j-1} \hdeg_I(M_j)$$
for each $1 \leq i \leq s-1$ and call them the homological torsions of $M$ with respect to $I$.
\end{defn}

The following Lemma 2.6 is a key for a proof of Theorem 1.3 (Theorem 3.3).

\begin{lem}\label{superficial2}
Let $M$ be a finitely generated $A$--module with $s = \dim_A M \ge 3$ and $I$  an $\m$-primary ideal of $A$. 
Then, for each $1 \leq i \leq s-2$, there exists a finite subset $\calF \subseteq \Proj R$ such that every $a \in I \setminus \bigcup_{\p \in \calF}[f^{-1}(\p) + \fkm I]$ is superficial for $M$ with respect to $I$, satisfying the inequality 
$$\rmT_I^i(M/aM) \le \rmT_I^i(M).$$
\end{lem}

\begin{proof}
See \cite[Lemma 2.8]{GO2}.
\end{proof}

Notice that 
$$ \hdeg_I(M)-\rmT_I^1(M)=\e_I^0(M)+\sum_{i=0}^{s-2}\binom{d-2}{i}\hdeg_I(M_i) $$
holds true.
We then have the following.

\begin{lem}\label{superficial3}
Let $M$ be a finitely generated $A$--module with $s = \dim_A M \ge 3$ and $I$ an $\m$-primary ideal of $A$.
Then, there exists a finite subset $\calF \subseteq \Proj R$ such that every $a \in I \setminus \bigcup_{\p \in \calF}[f^{-1}(\p) + \fkm I]$ is superficial for $M$ with respect to $I$, satisfying the inequality 
$$ \hdeg_I(M/aM) - \rmT_I^1(M/aM) \le \hdeg_I(M) - \rmT_I^1(M).$$
\end{lem}

\begin{proof}
See \cite[Lemma 2.7]{GO1}.
\end{proof}

\section{Bounds for the second Hilbert coefficients of parameters}

The purpose of this section is to estimate the second Hilbert coefficients of parameters in terms of the homological degrees and torsions of modules. 
For the estimation of $\e_Q^2(M)$, the key is the following result.

\begin{prop}\label{d=2}$(${\cite[Proposition 3.4]{GO3}}, \cite[Theorem 4.2]{GHV}$)$
Suppose that $M$ is a finitely generated $A$-module with $d=\dim_AM=2$ and $\depth_AM \geq 1$.
Let $Q=(a_1,a_2)$ be a parameter ideal of $A$ and assume that $a_1$ a superficial element for $M$ with respect to $Q$.
Then we have
$$-\ell_A(\H_{\m}^1(M)) \leq \e_Q^2(M) \leq 0 $$
and the following conditions are equivalent:
\begin{itemize}
\item[$(1)$] $\g_s(Q;M)=0$,
\item[$(2)$] $\e_Q^2(M)=0$,
\item[$(3)$] $a_1,a_2$ forms a $d$-sequence on $M$,
\item[$(4)$] $a_1^{\ell},a_2^{\ell}$ forms a $d$-sequence on $M$ for all $\ell \geq 1$.
\end{itemize}
\end{prop}

\begin{proof}
Since $A$ is complete, there exists a surjective homomorphism $\varphi:B \to A$ of rings, where $B$ is a Cohen-Macaulay complete local ring with $\dim B=\dim A$ and a system $\alpha_1,\alpha_2$ of parameters of $B$ such that $\varphi(\alpha_i)=a_i$ for $i=1,2$.
Therefore, passing to the ring $B$, we may assume that $A$ is a Cohen-Macaulay ring. Let $C=A \ltimes M$ denote the idealization of $M$ over $A$.
Then $C$ is a Noetherian local ring with maximal ideal $\n=\m \ltimes M$ and $\dim C= d$. We have
$$\ell_A(C/Q^{n+1}C) = \ell_A(A/Q^{n+1}) + \ell_A(M/Q^{n+1}M)$$
for all $n \geq 0$ and hence
\begin{eqnarray*}
\ell_A(M/Q^{n+1}M) &=&  \ell_C(C/Q^{n+1}C) - \ell_A(A/Q^{n+1})\\
&=& \left\{\e_{QC}^0(C)\binom{n+2}{2}-\e_{QC}^1(C)\binom{n+1}{1}+\e_{QC}^2(C) \right\} - \ell_A(A/Q)\binom{n+2}{2}\\
&=& \{\e_{QC}^0(C)-\ell_A(A/Q)\}\binom{n+2}{2}-\e_{QC}^1(C)\binom{n+1}{1}+\e_{QC}^2(C)
\end{eqnarray*}
for all $n \gg 0$ because the ring $A$ is Cohen-Macaulay.
Therefore $$-\ell_A(\H_{\m}^1(M)) \leq \e_{Q}^2(M) \leq 0$$ because $\e_Q^2(M)=\e_{QC}^2(C)$, $\H_{\m}^1(M) \cong \H_{\n}^1(C)$ as the ring $A$ is Cohen-Macaulay, and $-\ell_C(\H_{\n}^1(C)) \leq \e_{QC}^2(C) \leq 0$ by \cite[Proposition 3.4]{GO3}.

Let us consider the second assertion.

$(2) \Rightarrow (4)$
Since $\e_Q^2(C)=\e_Q^2(M)=0$, $a_1^{\ell}, a_2^{\ell}$ forms a $d$-sequence on $C$ for all $\ell \geq 1$ by Proposition \cite[Proposition 3.4]{GO3}.
Therefore $a_1^{\ell}, a_2^{\ell}$ forms a $d$-sequence on $M$ for all $\ell \geq 1$, because $A$ is a Cohen-Macaulay ring.

$(4) \Rightarrow (3)$
This is clear.

$(3) \Rightarrow (1)$ and $(2)$
We have $\ell_A(M/QM)=\e_Q^0(M)-\e_Q^1(M)+\e_Q^2(M)$ and $\e_Q^2(M)=0$ (\cite[Proposition 3.7]{GO2}).
Hence $\g_s(Q;M)=\ell_A(M/QM)-\e_Q^0(M)+\e_Q^1(M)=0$.

$(1) \Rightarrow (3)$
By \cite[Theorem 1.3]{GO1}.
\end{proof}

Passing to $M/\H_{\m}^0(M)$ we have the following inequalities.

\begin{cor}\label{d=2'}$($c.f. \cite[Corollary 3.5]{GO3}$)$
Suppose that $d=2$ and let $Q$ be a parameter ideal of $A$.
Then
$$\ell_A(\H_{\m}^0(M))-\ell_A(\H_{\m}^1(M)) \leq \e_Q^2(M) \leq \ell_A(\H_{\m}^0(M))$$
for every finitely generated $A$-module $M$ with $d=\dim_AM$.
\end{cor}

We then have the following.
Recall that a finitely generated $A$-module $M$ is said to be unmixed, if $\dim A/\p=\dim_AM$ for all $\p \in \Ass_{A}M$.

\begin{thm}\label{e2}
Let $M$ be a finitely generated $A$-module with $d =\dim_AM \ge 3$.
Then we have the following:
\begin{itemize}
\item[(1)] we have
$$\rme_Q^2(M) \leq \rmT_Q^2(M)$$
for every parameter ideals $Q$ of $A$.
\item[(2)] We have $$-\sum_{j=2}^{d-1} \binom{d-3}{j-2} \hdeg_Q(M_j) \le \rme_Q^2(M)$$ for every parameter ideals $Q$ of $A$, if $M$ is unmixed.
\end{itemize}
\end{thm}

We divide the proof of Theorem 3.3 into a few steps.
Let us begin with the following.

\begin{proof}[Proof of Theorem 3.3 (1)]
We proceed by induction on $d$.
Let $M' = M/\rmH_\m^0(M)$.
Then, since $\e_Q^2(M) = \e_Q^2(M')$ and $\rmT_Q^2(M) = \rmT_Q^2(M')$, to see that $\rme_Q^2(M) \leq \rmT_Q^2(M)$, we may assume, passing to $M'$, that $\depth_AM >0$. 
Suppose that $d=3$. Choose $a \in Q \setminus \fkm Q$ so that $a$ is superficial for $M$ and $M_1$ with respect to $Q$ and $\hdeg_Q(M_1 / aM_1) \le \hdeg_Q (M_1)$ (Lemma 2.4). Set $\overline{M} = M/aM$. Then since $a$ is $M$--regular, we get the exact sequence
$$ 0 \to \H_{\m}^0(\overline{M}) \to \H_{\m}^1(M) \overset{a}{\to} \H_{\m}^1(M) $$
of local cohomology modules. Taking the Matlis dual, we get an isomorphism $M_1/aM_1 \cong \overline{M}_0$ and hence, because $\e_Q^2(\overline{M}) \leq \ell_A(\H_{\m}^0(\overline{M}))$ by Corollary 3.2, we have
\begin{eqnarray*}
\e_Q^2(M) = \e_Q^2(\overline{M}) &\leq& \ell_A(\H_{\m}^0(\overline{M}))\\
&=& \hdeg_Q(\overline{M}_0)\\
&=& \hdeg_Q(M_1/aM_1)\\
&\leq& \hdeg_Q(M_1) = \rmT_Q^2(M).
\end{eqnarray*}
Suppose that $d \geq 4$ and that our assertion holds true for $d-1$. Choose $a \in Q \setminus \fkm Q$ so that $a$ is superficial for $M$ with respect to $Q$ and $\rmT_Q^2(\overline{M}) \le \rmT_Q^2(M)$ (Lemma 2.6). Then the hypothesis of induction on $d$ shows
$$\rme_Q^2(M)=\rme_Q^2(\overline{M}) \leq \rmT_Q^2(\overline{M}) \leq \rmT_Q^2(M),$$ 
as wanted.
\end{proof}

To prove Theorem 3.3 (2), we need the following.

\begin{prop}\label{unmixed}$(${\cite[Theorem 2.5]{GhGHOPV2}}$)$
Let $M$ be a finitely generated $A$-module with $d = \dim_AM$.
Suppose that $M$ is unmixed.
Then there exist a surjective homomorphism $B \to A$ of rings such that $B$ is a Gorenstein complete local ring with $\dim B=\dim A$ and an exact sequence
$$0 \to M \to F \to X \to 0$$
of $B$-modules with $F$ finitely generated and free.
\end{prop}

Thanks to Proposition 3.4, we get the following.

\begin{cor}\label{GNa}$(${\cite[Lemma 3.1]{GNa}}$)$
Let $M$ be a finitely generated $A$-module with $d = \dim_AM\geq 2$. If $M$ is unmixed, then $\H_{\m}^1(M)$ is finitely generated.
\end{cor}

Let $\Assh_AM=\{\p \in \Ass_AM \ | \ \dim A/\p=\dim_A M \}$ and let $(0)=\bigcap_{\p \in \Ass_AM}\operatorname{I}(\p)$ be a primary decomposition of $(0)$ in $M$, where for each $\p \in \Ass_AM$, $\operatorname{I}(\p)$ denotes a $\p$-primary submodule of $M$.
Set
$$\rmU_M(0)=\bigcap_{\p \in \Assh_AM}\operatorname{I}(\p)$$
and call it the unmixed component of $(0)$ in $M$.

We are now in a position to prove Theorem 3.3 (2).

\begin{proof}[Proof of Theorem 3.3 (2)]
By Proposition 3.4 we may assume that $A$ is a Gorenstein local ring and that there exists an exact sequence
$$ 0 \to M \overset{\varphi}{\to} F \to X \to 0 \ \ \ \ \ (\sharp)$$
of $A$-modules with $F$ a finitely generated free $A$-module and $X=\Coker \varphi$.

Since the residue class field $A/\m$ of $A$ is infinite, we may now choose an element $a \in Q \backslash \m Q$ so that $a$ is superficial for $M$, $F$, $X$, and $M_j$ with respect to $Q$ and $\hdeg_Q(M_j/aM_j) \leq \hdeg_Q(M_j)$ for all $2 \leq j \leq d-1$ (Lemma 2.4).
We set $\overline{M}=M/aM$.

We consider the exact sequence
$$ 0 \to [(0):_X a] \to \overline{M} \overset{\overline{\varphi}}{\to} F/aF \to X/aX \to 0$$
of $A$-modules obtained by exact sequence $(\sharp)$, where $\overline{\varphi}=A/(a) \otimes \varphi$.
Set $L=\Im \overline{\varphi}$.
Then since $L$ is unmixed with $\dim_A L=d-1$ and $\ell_A([(0):_X a])< \infty$, we get $$[(0):_X a] \cong \H_{\m}^0(\overline{M}) \cong \rmU_{\overline{M}}(0)$$ where $U=\rmU_{\overline{M}}(0)$ denotes the unmixed component of $(0)$ in $\overline{M}$.
Hence $\hdeg_Q(\overline{M}_j)=\hdeg_Q(L_j)$ for all $1 \leq j \leq d-2$.

We proceed by induction on $d$.
Suppose that $d=3$.
Consider the long exact sequence
$$ 0 \to \H_{\m}^0(\overline{M}) \to \H_{\m}^1(M) \overset{a}{\to} \H_{\m}^1(M) \to \H_{\m}^1(\overline{M}) \to \H_{\m}^2(M) \overset{a}{\to} \H_{\m}^2(M) $$
of local cohomology modules induced by the exact sequence
$$ 0 \to M \overset{a}{\to} M \to \overline{M} \to 0.$$
Because $L \cong \overline{M}/U$, $\ell_A(U) < \infty$, and $L$ is unmixed, $\H_{\m}^1(\overline{M}) \cong \H_{\m}^1(L)$ is finitely generated by Corollary 3.5.
Then, taking the Matlis dual of the above long exact sequence, we get exact sequences
$$ 0 \to M_2/aM_2 \to \overline{M}_1 \to [(0):_{M_1}a] \to 0 $$
and
$$ 0 \to [(0):_{M_1}a] \to M_1 \overset{a}{\to} M_1 \to \overline{M}_0 \to 0. $$Hence we have $$\ell_A(\H_{\m}^1(\overline{M}))=\ell_A(\overline{M}_1) = \ell_A([(0):_{M_1}a])+\ell_A(M_2/aM_2)$$ and $\ell_A(\overline{M}_0)=\ell_A([(0):_{M_1}a])$ because $\H_{\m}^1(M)$ is finitely generated by Corollary 3.5.
Therefore, by Corollary 3.2, we get
\begin{eqnarray*}
\e_Q^2(M)=\e_Q^2(\overline{M})&\geq&\ell_A(\H_{\m}^0(\overline{M}))-\ell_A(\H_{\m}^1(\overline{M}))\\
&=& \ell_A(\overline{M}_0)-\{\ell_A([(0):_{M_1}a])+\ell_A(M_2/aM_2)\}\\
&=& -\hdeg_Q(M_2/aM_2)\\
&\geq& -\hdeg_Q(M_2)
\end{eqnarray*}
as required.

Suppose that $d \geq 4$ and that our assertion holds true for $d-1$.
Consider the exact sequences
$$ \H_{\m}^j(M) \overset{a}{\to}  \H_{\m}^j(M) \to \H_{\m}^j(\overline{M}) \to \H_{\m}^{j+1}(M) \overset{a}{\to} \H_{\m}^{j+1}(M)$$
of local cohomology modules for $2 \leq j \leq d-2$ induced by the exact sequence
$$ 0 \to M \overset{a}{\to} M \to \overline{M} \to 0.$$
Then, taking the Matlis dual of the above exact sequences, we get exact sequences
$$ 0 \to M_{j+1}/aM_{j+1} \to \overline{M}_j \to [(0):_{M_j}a] \to 0 $$
and embeddings
$$ 0 \to [(0):_{M_j}a] \to M_j $$
for all $2 \leq j \leq d-2$.
Then, by Lemma 2.3, we have
\begin{eqnarray*}
\hdeg_Q(\overline{M}_j) &\leq& \hdeg_Q([(0):_{M_j}a])+\hdeg_Q(M_{j+1}/aM_{j+1})\\
&\leq& \hdeg_Q(M_j)+\hdeg_Q(M_{j+1})
\end{eqnarray*}
for all $2 \leq j \leq d-2$.
Because $a$ is superficial for $M$ with respect to $Q$ and $L \cong \overline{M}/U$ with $\ell_A(U) < \infty$, we see $\e_Q^2(M)=\e_Q^2(\overline{M})=\e_Q^2(L)$, and $\H_{\m}^j(\overline{M}) \cong \H_{\m}^j(L)$ for all $j \geq 1$.
Consequently, by the hypothesis of induction on $d$, we get
\begin{eqnarray*}
\e_Q^2(M)=\e_Q^2(\overline{M})&=&\e_Q^2(L)\\
&\geq& -\sum_{j=2}^{d-2}\binom{d-4}{d-2}\hdeg_Q(L_j)\\
&\geq& -\sum_{j=2}^{d-2}\binom{d-4}{d-2}\hdeg_Q(\overline{M}_j)\\
&\geq& -\sum_{j=2}^{d-2}\binom{d-4}{d-2}\{\hdeg_Q(M_j)+\hdeg_Q(M_{j+1})\}\\
&=& -\sum_{j=2}^{d-1}\binom{d-3}{d-2}\hdeg_Q(M_j)
\end{eqnarray*}
as required.
This completes the proof of Theorem 3.3.
\end{proof}

As a direct consequence of Theorem 3.3 we have the following, where, for an ideal $J$ in $A$, $\overline{J}$ denotes the integral closure of $J$.

\begin{cor}\label{finite}
Let $M$ be a finitely generated $A$-module with $d=\dim_A M \ge 2$ and suppose that $M$ is unmixed.
Let $Q$ be a parameter ideal in $A$. 
Then, the set
$$
\Lambda_Q^2(M)=\{\rme_\fkq^2(M) \mid \mbox{$\fkq$ is a parameter ideal in $A$ such that $\overline{\fkq} = \overline{Q}$}\}
$$
is finite.
\end{cor}

\begin{proof}
For each parameter ideal $\q$ of $A$, $-\ell_A(\H_{\m}^1(M)) \leq \e_{\q}^2(M) \leq 0$ if $d=2$ by Proposition 3.1, and $$-\sum_{j=2}^{d-1}\binom{d-3}{j-2}\hdeg_{\q}(M_j) \leq \e_{\q}^2(M) \leq \rmT_{\q}^2(M)$$ if $d \geq 3$ by Theorem 3.3.
Therefore the result follows, because $\ell_A(\H_{\m}^1(M)) < \infty$ by Corollary 3.5 and $\rmT_{\q}^2(M)$, $\hdeg_{\q}(M_j)$ depend only on the integral closure $\overline{\q}$ of $\q$.
\end{proof}

The following example shows that, the inequality $\e_Q^2(M) \geq -\sum_{j=2}^{d-1}\binom{d-3}{j-2}\hdeg_Q(M_j)$ in Theorem 3.3 (2) does not hold true in general, unless $M$ is unmixed,

\begin{ex}\label{ex}
Let $S$ be a complete regular local ring with maximal ideal $\n$, $d=\dim S=3$, and infinite residue class field $S/\n$. 
Let $\fkn = (X,Y,Z)$ and $\ell \geq 1$ be integers.
We set
$$ A=S \ltimes S/(Z^{\ell}) $$
denotes the idealization of $S/(Z^{\ell})$ over $S$ and let $Q=\n A$.
Then we have the following.
\begin{itemize}
\item[(1)] $A$ is mixed with maximal ideal $\m=\n \times [S/(Z^{\ell})]$, $\dim A=3$, and $\depth A=2$,
\item[(2)] $\e_Q^0(A)=1$, $\e_Q^1(A)=-\ell$, $\e_Q^2(A)=-\binom{\ell}{2}$, and $\e_Q^3(A)=-\binom{\ell}{3}$,
\item[(3)] $\hdeg_Q(A_2)=\ell$.
\item[(4)] Hence $-\hdeg_Q(A_2) > \e_Q^2(A)$, if $\ell \geq 4$.
\end{itemize}
\end{ex}

\begin{proof}
We set $\overline{S}=S/(Z^{\ell})$.
Since $S$ is a Regular local ring with $\dim S=3$, we have
\begin{eqnarray*}
\ell_A(A/Q^{n +1})  &=& \ell_S(S/\fkn^{n +1}) + \ell_S(\overline{S}/\fkn^{n+1}\overline{S})\\
&=& \binom{n +3}{3} + \left\{\rme_\fkn^0(\overline{S}) \binom{n +2}{2} -\rme_\fkn^1(\overline{S}) \binom{n +1}{1} + \rme_\fkn^2(\overline{S})\right\}
\end{eqnarray*}
for all $n \gg 0$

Because the Hilbert series $\H(\gr_{\n}(\overline{S}), \lambda)$ of the associated graded ring $\gr_{\n}(\overline{S})$ is given by
$$
\rmH(\gr_\n(\overline{S}) , \lambda) = \frac{1+\lambda + \cdots +\lambda^{\ell-1}}{(1-\lambda)^2},
$$
we get $\rme_\fkn^0(\overline{S}) = \ell$, $\rme_\fkn^1(\overline{S}) = \binom{\ell}{2}$, and $\rme_\fkn^2(\overline{S}) = \binom{\ell}{3}$. 
Therefore

$$
(-1)^i\e_Q^i(A) = \left\{
\begin{array}{lc}
1 & \mbox{if $i = 0$},\\
\rme_\fkn^0(\overline{S}) = \ell & \mbox{if $i = 1$},\\
-\rme_\fkn^1(\overline{S}) = -\binom{\ell}{2} & \mbox{if $i = 2$},\\
\rme_\fkn^2(\overline{S}) = \binom{\ell}{3}& \mbox{if $i = 3$}.
\end{array}
\right.
$$

On the other hand, we have 
$$
\hdeg_Q(A_2) = \hdeg_Q(\overline{S}) = \rme_{Q}^0(\overline{S})=\e_{\n}^0(\overline{S}) = \ell,
$$ because $\overline{S}$ is a Gorenstein ring and $\rmH_\fkm^2(A) \cong {}_p[\rmH_\fkn^2(\overline{S})]$ (here $p : A \to \overline{S}, p(a,x) = x$ denotes the projection).
\end{proof}

We close this section with the following example of parameter ideals $Q$ such that $\e_Q^2(M)=\rmT_{Q}^2(M)$ but $A$ is not a generalized Cohen-Macaulay ring.

\begin{ex}\label{ex1}$($\cite[Example 3.10]{GO2}$)$
Let $\ell \geq 2$ and $m \geq 1$ be integers.
Let $$S=k[[X_i,Y_i,Z_j \ | \ 1 \leq i \leq \ell, 1 \leq j \leq m]]$$
be the formal power series ring with $2\ell+m$ indeterminates over an infinite field $k$ and set $\a=(X_1,X_2,\ldots,X_{\ell})S$, $\b=(Y_1,Y_2,\ldots,Y_{\ell})S$. Let $$A=S/\a \cap \b,$$
$$\m=(x_i,y_i,z_j \ | \ 1 \leq i \leq \ell, 1 \leq j \leq m)A, ~~\text{and}$$ 
$$Q=(x_i-y_i \ | \ 1 \leq i \leq \ell)A+(z_j \ | \ 1 \leq j \leq m)A,$$
where $x_i$, $y_i$, and $z_j$ denote the images of $X_i$, $Y_i$, and $Z_j$ in $A$ respectively.
Then  $\m^2=Q\m$, whence $Q$ is a reduction of $\m$. We furthermore have the following:
\begin{itemize}
\item[$(1)$] $A$ is an unmixed local ring with $\dim A=\ell+m$, $\depth A=m+1$, and $\H_{\m}^{m+1}(A)$ is not  finitely generated.
\item[$(2)$] $\e_Q^2(A)=0$ if $\ell =2$, and $\e_Q^2(A)=1$ if $\ell \geq 3$.
\item[$(3)$] $\rmT_Q^2(A)=\binom{\ell+m-3}{m}$.
\item[$(4)$] Hence $\e_Q^2(A)=\rmT_Q^2(A)$, if $\ell=2, 3$, but $\e_Q^2(A) < \rmT_Q^2(A)$ if $\ell \geq 4$.
\end{itemize}
\end{ex}

\begin{proof}
Consider the exact sequence
$$ 0 \to A \to S/\a \times S/\b \to S/[\a+\b] \to 0$$
of $S$-modules. 
Then because $$S/\a \cong k[[Y_i, Z_j \ | \ 1 \leq i \leq \ell, 1 \leq j \leq m]],$$ 
$$S/\b \cong k[[X_i, Z_j \ | \ 1 \leq i \leq \ell, 1 \leq j \leq m]], \ \ \ \mbox{and}$$ 
$$S/[\a+\b] \cong k[[Z_j \ | \ 1 \leq j \leq m]],$$ 
we get $\dim A=\ell+m$, $\H_{\m}^{m+1}(A) \cong \H_{\m}^m(S/[\a+\b])$, and $\H_{\m}^j(A)=0$ for all $j \neq m+1$, $\ell+m$.
Hence we have $$\hdeg_Q(A_{m+1})=\hdeg_Q(S/[\a+\b])=\e_Q^0(S/[\a+\b])=\e_{\m}^0(S/[\a+\b])=1$$ 
and $\hdeg_Q(A_j)=0$ for all $0 \leq j \leq \ell+m-1$ such that  $j \neq m+1$.
Therefore we get
$$ \rmT_Q^2(A)=\sum_{j=1}^{\ell+m-2}\binom{\ell+m-3}{j-1}\hdeg_Q(A_j)=\binom{\ell+m-3}{m}.$$

On the other hand, we set $$B=S'/\a' \cap \b'$$ and $Q_0=(x_i-y_i \ | \ 1 \leq i \leq \ell)B$ where $S'=k[[X_i,Y_i \ | \ 1 \leq i \leq \ell]]$ be the formal power series ring, $\a'=(X_1,X_2,\cdots,X_{\ell})S'$, and $\b'=(Y_1,Y_2,\cdots,Y_{\ell})S'$.
Then we have $A=B[[Z_j \ | \ 1 \leq j \leq m]]$ and $Q=Q_0A+(z_j \ | \ 1 \leq j \leq m)A$.
Recall that $\gr_Q(A)=\gr_{Q_0}(B)[W_1',W_2',\cdots,W_m']$ forms the polynomial ring, where $W_j$'s are the initial forms of $z_j$'s.
Therefore $z_1,z_2,\cdots,z_m$ forms a superficial sequence for $A$ with respect to $Q$ so that we have $$\e_Q^2(A)=\e_{Q_0}^2(B)=\left\{
\begin{array}{lc}
0 & \mbox{if $\ell = 2$},\\
1 & \mbox{if $\ell \geq 3$},
\end{array}
\right.$$
because $B$ is a Buchsbaum ring with $\H_{\n}^1(B) \cong k$ and $\H_{\n}^i(B)=(0)$ for all $i \neq 1, \ell$ (\cite[Proposition 2.7]{SV}).
\end{proof}

\section{Relationship between the second Hilbert coefficients and the homological torsion of parameters}

The second Hilbert coefficients $\e_Q^2(M)$ of parameter ideals are bounded above by the homological torsion $\rmT_Q^2(M)$. 
It is now natural to ask what happens on the parameters $Q$ of $M$, once the equality $\e_Q^2(M)=\rmT_Q^2(M)$ is attained. 
Let $$\g_s(Q;M)=\ell_A(M/QM)-\e_Q^0(M)+\e_Q^1(M)$$ denotes the sectional genus of $M$ with respect to $Q$.
Then the main result of this section answers the question and is stated as follows.

\begin{thm}\label{main2}
Let $M$ be a finitely generated $A$-module with $d = \dim_AM \geq 3$ and suppose that $M$ is unmixed. 
Let $Q$ be a parameter ideal of $A$. 
Then the following conditions are equivalent:
\begin{itemize}
\item[$(1)$] $\g_s(Q;M)=\hdeg_Q(M)-\e_Q^0(M)-\rmT_Q^1(M)$,
\item[$(2)$] $\e_Q^2(M)=\rmT_Q^2(M)$.
\end{itemize}
When this is the case, we have the following$\mathrm{:}$
\begin{itemize}
\item[$(\mathrm{i})$] $(-1)^i\e_Q^i(M)=\rmT_Q^i(M)$ for $3 \leq i\leq d-1$ and $\e_Q^d(M)=0$,
\item[$(\mathrm{ii})$] $\ell_A(M/Q^{n+1}M)=\sum_{i=0}^d(-1)^i\e_Q^i(M)\binom{n+d-i}{d-i}$ for all $n \geq 0$,
\item[$(\mathrm{ii})$] there exist elements $a_1,a_2,\ldots,a_d \in A$ such that $Q=(a_1,a_2,\ldots,a_d)$ and $a_1,a_2,\ldots,a_d$ forms a $d$-sequence on $M$, and
\item[$(\mathrm{iii})$] $Q\H_{\m}^i(M)=(0)$ for all $1 \leq i \leq d-3$.
\end{itemize}
\end{thm}

In \cite[Theorem 3.1]{GO1}, we gave the characterization of parameter ideals $Q$ satisfying the equality $\g_s(Q;M)=\hdeg_Q(M)-\e_Q^0(M)-\rmT_Q^1(M)$ as follows.

\begin{thm}\label{genus}$($\cite[Theorem 1.3]{GO1}$)$
Let $M$ be a finitely generated $A$-module with $d = \dim_AM \geq 2$ and let $Q$ be a parameter ideal of $A$.
Then the following conditions are equivalent:
\begin{itemize}
\item[$(1)$] $\g_s(Q;M)=\hdeg_Q(M)-\e_Q^0(M)-\rmT_Q^1(M)$.
\item[$(2)$] The following two conditions are satisfied:
\begin{itemize}
\item[$(a)$]
$$
(-1)^i\e_Q^i(M) = \left\{
\begin{array}{cl}
\rmT_Q^i(M) & \mbox{if $2 \leq i \leq d-1$},\\
\vspace{0.5mm}\\
\ell_A(\H_{\m}^0(M)) & \mbox{if $i=d$}
\end{array}
\right.
$$
for all $2 \leq i \leq d$.
\item[$(b)$] $\ell_A(M/QM)-\sum_{i=0}^d(-1)^i\e_Q^i(M)=0.$
\end{itemize}
\end{itemize}
When this is the case, we have the following$\mathrm{:}$
\begin{itemize}
\item[$(\mathrm{i})$] there exist elements $a_1,a_2,\ldots,a_d \in A$ such that $Q=(a_1,a_2,\ldots,a_d)$ and $a_1,a_2,\ldots,a_d$ forms a $d$-sequence on $M$,
\item[$(\mathrm{ii})$] $$\ell_A(M/Q^{n+1}M)=\sum_{i=0}^d(-1)^i\e_Q^i(M)\binom{n+d-i}{d-i}$$ for all $n \geq 0$,
\item[$(\mathrm{iii})$] $QM \cap \H_{\m}^0(M)=(0)$, and $Q\H_{\m}^i(M)=(0)$ for all $1 \leq i \leq d-3$.
\end{itemize}
\end{thm}

Thanks to Theorem 4.2, the implication $(1) \Rightarrow (2)$ and the last assertions follow.
Thus we need to show the implication $(2) \Rightarrow (1)$ for the proof of Theorem 4.1.

To prove Theorem 4.1, we need some results which are concerned to the sectional genera $\g_s(Q;M)$ as follows.
See \cite{GO1} for the detailed proofs.

\begin{lem}\label{/a}
Let $M$ be a finitely generated $A$-module with $s=\dim_AM \geq 3$.
Let $Q$ be a parameter ideal for $M$ and assume that $a \in Q \backslash \m Q$ is a superficial element for $M$ with respect to $Q$.
Then $\g_s(Q;M)=\g_s(\overline{Q};\overline{M})$, where $\overline{M}=M/aM$ and $\overline{Q}=Q/(a)$.
\end{lem}

\begin{proof}
See \cite[Lemma 3.2]{GO1}.
\end{proof}

\begin{lem}\label{W}
Let $M$ be a finitely generated $A$-module with $d = \dim_AM \geq 2$ and let $Q$ be a parameter ideal of $A$.
Then $\g_s(Q;M)=\hdeg_Q(M)-\e_Q^0(M)-\rmT_Q^1(M)$ if and only if $\g_s(Q;M/\H_{\m}^0(M))=\hdeg_Q(M/\H_{\m}^0(M))-\e_Q^0(M/\H_{\m}^0(M))-\rmT_Q^1(M/\H_{\m}^0(M))$ and $QM \cap \H_{\m}^0(M)=(0)$.
\end{lem}

\begin{proof}
See \cite[Lemma 3.5]{GO1}.
\end{proof}

\begin{prop}\label{d-seq1}
Let $M$ be a finitely generated $A$-module with $d = \dim_AM \geq 2$ and $Q$ a parameter ideal of $A$.
Let $a_1 \in Q \backslash \m Q$ be a superficial element for $M$ with respect to $Q$ such that $\hdeg_Q(M/a_1M)-\rmT_Q^1(M/aM) \leq \hdeg_Q(M)-\rmT_Q^1(M)$.
Assume that $$\g_s(Q;M)=\hdeg_Q(M)-\e_Q^0(M)-\rmT_Q^1(M).$$
Then there exist elements $a_2,a_3,\ldots,a_d \in A$ such that $Q=(a_1, a_2,\ldots,a_d)$ and $a_1, a_2,\ldots,a_d$ forms a $d$-sequence on $M$.
\end{prop}

\begin{proof}
See \cite[Lemma 3.6]{GO1}.
\end{proof}

We also need the following, where
$$\rmU(N)=\bigcup_{\ell>0}[N:_M \m^{\ell}]$$ for each submodule $N$ of $M$.

\begin{prop}\label{cap}
Suppose that $A$ is a Cohen-Macaulay local ring with $d \geq 2$.
Let $M$ be a finitely generated $A$-module with $d= \dim_AM$.
Assume that there exists an exact sequence $0 \to M \to F \to X \to 0$ of $A$-modules with $F$ a finitely generated free $A$-module.
Let $Q=(a,b,a_3,\ldots,a_d)$ be a parameter ideal of $A$ and assume that $a$ is a superficial element for $M$, $F$, and $X$, $a, a_3, \cdots,a_d$ forms a $d$-sequence on $M/\rmU(bM)$, and $b\H_{\m}^1(M)=(0)$.
Then $$Q(M/aM) \cap \H_{\m}^0(M/aM)=(0).$$
\end{prop}

\begin{proof}
We set $\overline{M}=M/aM$, $\overline{M}'=M/bM$, and $L'=\overline{M}'/\H_{\m}^0(\overline{M}')=M/\rmU(bM)$.
Take $\alpha \in Q\overline{M} \cap \H_{\m}^0(\overline{M})$ and write $\alpha=\overline{x}$ with $x \in QM \cap \rmU(aM)$, where $\overline{x}$ denotes the image of $x$ in $\overline{M}$. 
Let us  consider the composite of the canonical maps
$$ \rho:M \to \overline{M}' \to L' \to L'/aL'.$$ Then $\m^{\ell}x \subseteq aM$ for all $\ell \gg 0$ and $x \in QM$.  Therefore 
$$ \rho(x) \in \H_{\m}^0(L'/aL') \cap (a,a_3,\ldots,a_d)(L'/aL')=(0),$$
because $a, a_3,\cdots,a_d$ forms a $d$-sequence on $L'$. 
Consequently, $x \in aM+\rmU(bM) \cap \rmU(aM)$.
Let us write $x=y+z$ with $y \in aM$ and $z \in \rmU(bM) \cap \rmU(aM)$. Then because $a$ and $b$ are $M$-regular, we have the embeddings $$ \rmU(aM)/aM=\H_{\m}^0(\overline{M}) \hookrightarrow \H_{\m}^1(M),$$
$$ \rmU(bM)/bM=\H_{\m}^0(\overline{M}') \hookrightarrow \H_{\m}^1(M),$$ so that, for some integer $\ell > 0$, $a^{\ell}\rmU(bM) \subseteq bM$ and $b\rmU(aM) \subseteq aM$, since $a^{\ell}\H_{\m}^1(M)=(0)$ as $\H_{\m}^1(M) \cong \H_{\m}^0(X)$ is finitely generated and $b\H_{\m}^1(M)=(0)$.
Therefore  $ a^{\ell}z \in bM$ and $bz \in aM$.
We now write 
$$ a^{\ell}z = bv \ \ \ \mbox{and} \ \ \ bz =aw$$
with $v,w \in M$. 
Then $v \in [a^{\ell+1}M:_Mb^2]$, since $a^{\ell}bz=b^2v=a^{\ell+1}w$.

\begin{claim}
$[a^{\ell+1}M:_Mb^2] = [a^{\ell+1}M:_M b]$.
\end{claim}

\noindent
{\it{Proof of Claim 1.}}
Tensoring exact sequence 
$$0 \to M \overset{\varphi}{\to} F \to X \to 0$$
by $A/(a^{\ell+1})$, we get the exact sequence
$$ 0 \to [(0):_X a^{\ell+1}] \to M/a^{\ell+1}M \overset{\widetilde{\varphi}}{\to} F/a^{\ell+1}F \to X/a^{\ell+1}X \to 0,$$
where $\widetilde{\varphi}=A/(a^{\ell+1}) \otimes \varphi$. Since $\ell_A([(0):_Xa]) < \infty$, $[(0):_X a]_{\p}=(0)$ for all $\p \in \Spec A \backslash \{\m\}$.
Hence $a$ is $X_{\p}$-regular, so that $\ell_A([(0):_Ma^{\ell+1}])< \infty$.
Therefore because $\depth_AF/a^{\ell+1}F >0$, we get an isomorphism $$[(0):_X a^{\ell+1}] \cong \H_{\m}^0(M/a^{\ell+1}M).$$
Take $\xi \in [a^{\ell+1}M:_Mb^2]$ and let $\overline{\xi}$ denotes the image of $\xi$ in $M/a^{\ell+1}M$.
Then $b^2\widetilde{\varphi}(\overline{\xi})=\widetilde{\varphi}(b^2\overline{\xi})=0$ in $F/a^{\ell+1}F$, whence $\widetilde{\varphi}(\overline{\xi})=0$, because $a^{\ell+1},b^2$ forms an $F$-regular sequence.
Therefore $$\overline{\xi} \in \ker \widetilde{\varphi} \cong \H_{\m}^0(M/a^{\ell+1}M) \hookrightarrow \H_{\m}^1(M)$$
and hence $b \overline{\xi}=0$ in $M/a^{\ell+1}M$, because $b \H_{\m}^1(M)=(0)$.
Thus $b\xi \in a^{\ell+1}M$, so that $\xi \in [a^{\ell+1}M:_M b]$.
Consequently $[a^{\ell+1}M:_Mb^2] \subseteq [a^{\ell+1}M:_M b]$, which proves Claim \ref{claim}.\\

We have $v \in [a^{\ell+1}M:_M b]$ by Claim 1.
Hence $bv \in a^{\ell+1}M$. 
Then $z \in aM$, since $a^{\ell}bz=b^2v \in a^{\ell+1}bM$ and $\depth_AM >0$. Therefore $x=y+z \in aM$, so that $\alpha=\overline{x}=0$ in $\overline{M}$. Thus $Q\overline{M} \cap \H_{\m}^0(\overline{M})=(0)$, which proves Proposition 4.6.
\end{proof}

We are now in a position to prove Theorem 4.1.

\begin{proof}[Proof of Theorem 4.1]
Thanks to \cite[Theorem 1.3]{GO1}, we have only to show the implication $(2) \Rightarrow (1)$. 
By Proposition 3.4 we may assume that $A$ is a Gorenstein local ring and that there exists an exact sequence
$$0 \to M \overset{\varphi}{\to} F \to X \to 0 \ \ \ \ (\sharp)$$
of $A$-modules with $F$ a finitely generated free $A$-module and $X =\operatorname{Coker} \varphi$.

We proceed by induction on $d$. 
Suppose that $d=3$ and let $Q=(a,b,c)$. Since the residue class field $A/\m$ of $A$ is infinite, we may choose the element  $a$ is superficial for $M$, $F$, and $X$ with respect to $Q$, and set $\overline{Q}=Q/(a)$, $\overline{M}=M/aM$, $W=\H_{\m}^0(\overline{M})$, and $L=\overline{M}/W$.
Then by the exact sequence
$$ 0 \to W \to \H_{\m}^1(M) \overset{\hat{a}}{\to} \H_{\m}^1(M) $$
of local cohomology modules, we have $$\ell_A(W) = \ell_A([(0):_{\H_{\m}^1(M)}a]).$$
Then
\begin{eqnarray*}
\e_Q^2(M)=\e_Q^2(\overline{M})&=&\e_Q^2(L)+\ell_A(W)\\
&\leq& \ell_A([(0):_{\H_{\m}^1(M)}a])\\
&\leq& \ell_A(\H_{\m}^1(M))\\
&=& \hdeg_Q(M_1)\\
&=& \rmT_Q^2(M)=\e_Q^2(M)
\end{eqnarray*}
because $\e_Q^2(M)=\e_Q^2(\overline{M})=\e_Q^2(L)+\ell_A(W)$, and $\e_Q^2(L) \leq 0$ by Proposition 3.1.
Hence $\e_Q^2(L) = 0$, $a\H_{\m}^1(M)=(0)$, and $\hdeg_Q(\overline{M}_0)=\ell_A(W)=\hdeg_Q(M_1)$.
Then, by Proposition 3.1, we get $$\g_s(\overline{Q};L)=0=\hdeg_Q(L)-\e_Q^0(L)-\rmT_Q^1(L)$$
because $\e_Q^2(L)=0$.
On the other hand, we have $\g_s(Q;M)=\g_s(\overline{Q};\overline{M})$ by Lemma 4.3 and
\begin{eqnarray*}
\hdeg_Q(M)-\e_Q^0(M)-\rmT_Q^1(M)&=&\hdeg_Q(M_1)\\
&=& \hdeg_Q(\overline{M}_0)\\
&=& \hdeg_{\overline{Q}}(\overline{M})-\e_{\overline{Q}}^0(\overline{M})-\rmT_{\overline{Q}}^1(\overline{M}).
\end{eqnarray*}
Therefore, thanks to Lemma 4.4, to prove $\g_s(Q;M)=\hdeg_Q(M)-\e_Q^0(M)-\rmT_Q^1(M)$, it is enough to show that $Q\overline{M} \cap W=(0)$.

Let us choose the element $b$ is superficial for $M$ with respect to $Q$, and set $\overline{M}'=M/bM$ and $L'=\overline{M}'/\H_{\m}^0(\overline{M}')$.
Then by the same argument as above, $\e_Q^2(L')=0$ and $b\H_{\m}^1(M)=(0)$.
We now choose the element $a$ to be superficial also for $L'$ with respect to $\overline{Q}'$.
Then, thanks to Proposition 3.1, $a,c$ forms a $d$-sequence on $L'$, because $\e_Q^2(L')=0$.
Therefore $Q\overline{M} \cap W=(0)$ by Proposition 4.6.
Thus, by Lemma 4.4,
$\g_s(Q;M)=\hdeg_Q(M)-\e_Q^0(M)-\rmT_Q^1(M)$ as required.

Assume that $d \geq 4$ and that our assertion holds true for $d-1$.
Since the residue class field $A/\m$ of $A$ is infinite, we may now choose an element $a \in Q \backslash \m Q$ so that $a$ is superficial for $M$, $F$, $X$, and $M_j$ with respect to $Q$ and $\hdeg_Q(M_j/aM_j) \leq \hdeg_Q(M_j)$ for all $1 \leq j \leq d-2$.
Set $\overline{M}=M/aM$ and $\overline{Q}=Q/(a)$.
Then, by the same argument as is in the proof of Theorem 3.3 (2), we have
$$ \ell_A([(0):_{M_j} a]) \leq \hdeg_Q(M_j)  \ \ \ \text{and}$$
$$ \hdeg_Q(\overline{M_j}) \leq \ell_A([(0):_{M_j} a]) +\hdeg_Q(M_{j+1}/aM_{j+1}) $$
for all $1 \leq j \leq d-3$.

We consider the exact sequence
$$ 0 \to [(0):_Xa] \to \overline{M} \overset{\overline{\varphi}}{\to} F/aF \to X/aX \to 0 $$
of $A$-modules obtained by  exact sequence $(\sharp)$,  where $\overline{\varphi}=A/(a) \otimes \varphi$.
Set $L=\Im \overline{\varphi}$. Then since $L$ is unmixed with $\dim_AL=d-1$ and $\ell_A([(0):_Xa]) < \infty$, we get $$ [(0):_Xa] \cong \H_{\m}^0(\overline{M})=\rmU_{\overline{M}}(0), $$
where $U=\rmU_{\overline{M}}(0)$ denotes the unmixed component of $(0)$ in $\overline{M}$. Consequently, because $a$ is superficial for $M$ with respect to $Q$ and $L \cong \overline{M}/U$ with $\ell_A(U) < \infty$, we see $\e_Q^i(M)=\e_Q^i(\overline{M})=\e_Q^i(L)$ for $i=0, 1, 2$ and $\H_{\m}^j(L) \cong \H_{\m}^j(\overline{M})$ for all $j \geq 1$.
Hence $\hdeg_Q(L_j)=\hdeg_Q(\overline{M}_j)$ for all $1 \leq j \leq d-3$. Therefore 
\begin{eqnarray*}
\e_Q^2(M)=\e_Q^2(\overline{M})=\e_Q^2(L) &\leq& \rmT_Q^2(L)\\
&=& \sum_{j=1}^{d-3}\binom{d-4}{j-1}\hdeg_Q(L_j)\\
&=& \sum_{j=1}^{d-3}\binom{d-4}{j-1}\hdeg_Q(\overline{M}_j)\\
&\leq& \sum_{j=1}^{d-3}\binom{d-4}{j-1}\{ \ell_A([(0):_{M_j} a]) +\hdeg_Q(M_{j+1}/aM_{j+1}) \}\\
&\leq& \sum_{j=1}^{d-3}\binom{d-4}{j-1}\{\hdeg_Q(M_j)+\hdeg_Q(M_{j+1})\}\\
&=& \sum_{j=1}^{d-2}\binom{d-3}{j-1}\hdeg_Q(M_j)\\
&=& \rmT_Q^2(M)=\e_Q^2(M),
\end{eqnarray*}
because $\e_Q^2(L) \leq \rmT_{Q}^2(L)$ by Theorem 3.3. 
Thus $\e_Q^2(L)=\rmT_Q^2(L)$, $\hdeg_Q(\overline{M}_j)=\hdeg_Q(M_j)+\hdeg_Q(M_{j+1})$, and $aM_j=(0)$ for all $1 \leq j \leq d-3$, so that the hypothesis of induction on $d$ yields $\g_s(\overline{Q}:L)=\hdeg_{Q}(L)-\e_{Q}^0(L)-\rmT_Q^1(L)$.
We also have $\g_s(Q;M)=\g_s(\overline{Q};\overline{M})$ by Lemma 4.3, and
\begin{eqnarray*}
\hdeg_{Q}(\overline{M})-\e_{Q}^0(\overline{M})-\rmT_Q^1(\overline{M})&=& \sum_{j=0}^{d-3}\binom{d-3}{j}\hdeg_Q(\overline{M}_j)\\
&=& \sum_{j=0}^{d-3}\binom{d-3}{j}\{\hdeg_Q(M_j)+\hdeg_Q(M_{j+1})\}\\
&=& \sum_{j=0}^{d-2}\binom{d-2}{j}\hdeg_Q(M_j)\\
&=& \hdeg_Q(M)-\e_Q^0(M)-\rmT_Q^1(M). \ \ \ \ \ \ \ \ \ \ \ \ \ (\dagger_1)
\end{eqnarray*}
Thus, thanks to Lemma 4.4, it is enough to show that $Q\overline{M} \cap U=(0)$.

Let us choose an element $b \in Q \backslash \m Q$ so that $b$ is  superficial for $M$, $F$, $X$, and $M_j$ with respect to $Q$, $\hdeg_Q(M_j/bM_j) \leq \hdeg_Q(M_i)$ for all $1 \leq j \leq d-2$, and $a$, $b$ forms a part of a minimal system of generators of $Q$.
We set $\overline{M}'=M/bM$ and $\overline{Q}'=Q/(b)$. Then, tensoring $(\sharp)$ by $\overline{A}'=A/(b)$, we get the exact sequence
$$ 0 \to [(0):_Xb] \to \overline{M}' \overset{\overline{\varphi}'}{\to} F/bF \to X/bX \to 0, $$ 
where $\overline{\varphi}'=A/(b) \otimes \varphi$. Set $L'=\Im \overline{\varphi}'$. Then because $L'$ is unmixed with $\dim_AL'=d-1$ and $\ell_A([(0):_X b]) < \infty$, we have
$$ [(0):_X b] \cong \H_{\m}^0(\overline{M}')=\rmU_{\overline{M}'}(0),$$
where $U'=\rmU_{\overline{M}'}(0)$ is the unmixed component of $(0)$ in $\overline{M}'$. Consequently by the same argument as above, $\e_Q^2(L')=\rmT_Q^2(L')$ and $bM_i=(0)$ for all $1 \leq i \leq d-3$, so that thanks to the hypothesis of induction on $d$, we get $\g_s(\overline{Q}';L')=\hdeg_{\overline{Q}'}(L')-\e_{\overline{Q}'}^0(L')-\rmT_{\overline{Q}'}^1(L')$.

We now choose the element $a \in Q \setminus \m Q$ to be superficial also for $L'$ with respect to $\overline{Q}'$ and $\hdeg_{\overline{Q}'}(L'/aL')-\rmT_{\overline{Q}'}^1(L'/aL') \leq \hdeg_{\overline{Q}'}(L')-\rmT_{\overline{Q}'}^1(L')$ (Lemma 2.7).
Then by Proposition 4.5 there exist elements $a_3,a_4,\ldots,a_d \in A$ such that $\overline{Q}'=(a,a_3,\ldots,a_d)\overline{A}'$ and $a,a_3,a_4,\ldots,a_d$ forms a $d$-sequence on $L'$, because $\g_s(\overline{Q}';L')=\hdeg_{\overline{Q}'}(L')-\e_{\overline{Q}'}^0(L')-\rmT_{\overline{Q}'}^1(L')$.
Therefore, by Proposition 4.6, we get $Q\overline{M} \cap U=(0)$.
Hence the required equality $\g_s(Q;M)=\hdeg_Q(M)-\e_Q^0(M)-\rmT_Q^1(M)$ follows by Lemma 4.4, which completes the proof of Theorem 4.1 as well as the proof of the implication $(2) \Rightarrow (1)$.
\end{proof}

The following example shows that the implication $(2) \Rightarrow (1)$ does not hold true in general, unless $M$ is unmixed.

\begin{ex}\label{ex2}
Let $S$ be a complete regular local ring with maximal ideal $\n$, $\dim S = 4$, and infinite residue class field $S/\n$. 
Let $\fkn = (X,Y,Z,W)$ and $\ell \geq 1$ be integers.
We set $$A=S/(X) \cap (Y^{\ell}, Z,W).$$ Let $\m=(x,y,z,w)A$ be the maximal ideal of $A$ and $Q=(x-y,x-z, x-w)A$, where $x$, $y$, $z$, and $w$ denote the images of $X$, $Y$, $Z$, and $W$ in $A$, respectively. Then, since $\m^{\ell+1}=Q\m^{\ell}$, $Q$ is a reduction of $\m$. We furthermore have the following:

\begin{itemize}
\item[$(1)$] $A$ is mixed with $\dim A=3$ and $\depth A=1$,
\item[$(2)$] $\ell_A(A/Q)=2$, $\e_Q^0(A)=1$, $\e_Q^1(A)=0$, $\e_Q^2(A)=\ell$, and $\e_Q^3(A)=\binom{\ell}{2}$. Hence $\g_s(Q;A)=1$.
\item[$(3)$] $\hdeg_Q(A)=2\ell+1$ and $\rmT_Q^1(A)=\rmT_Q^2(A)=\ell$.
\item[$(4)$] $\e_Q^2(A)=\rmT_Q^2(A)$ for all $\ell \geq 1$.
\item[$(5)$] $\g_s(Q;A)= \hdeg_Q(A)-\e_Q^0(A)-\rmT_Q^1(A)$ if $\ell=1$, but $\g_s(Q;A)< \hdeg_Q(A)-\e_Q^0(A)-\rmT_Q^1(A)$ if $\ell \geq 2$.
\end{itemize}
\end{ex}

\begin{proof}
Consider the canonical exact sequence
$$ 0 \to xA \to A \to A/xA \to 0.$$
Set $\a=(y^{\ell},z, w)A$. Then $U=xA ~(\cong A/\a$) is the unmixed component of $(0)$ in $A$.
Set $B=A/xA$.
Then since $B$ is a regular local ring with $\dim B=3$ and $QB=\m B$, we have
\begin{eqnarray*}
\ell_A(A/Q^{n+1}) &=& \ell_A(B/\m^{n+1}B)+\ell_A(U/Q^{n+1}U)\\
&=& \binom{n+3}{3}+\left[ \e_Q^0(U)\binom{n+1}{1}-\e_Q^1(U) \right]
\end{eqnarray*}
for all $n \gg 0$.

Because the Hilbert series $\rmH(\gr_\m(A/\a), \lambda)$ of the associated graded ring $\gr_\m(A/\a)$ is given by
$$
\rmH(\gr_\m(A/\a) , \lambda) = \frac{1+\lambda + \cdots +\lambda^{\ell-1}}{1-\lambda}
$$
and $Q{\cdot}(A/\a)=\m{\cdot}(A/\a)$, we have $\e_Q^0(U)=\e_{\m}^0(A/\a) = \ell$ and $\e_Q^1(U)=\e_\m^1(A/\a) = \binom{\ell}{2}$.
Hence
$$
(-1)^i\e_Q^i(A) = \left\{
\begin{array}{lc}
1 & \mbox{if $i = 0$},\\
0 & \mbox{if $i = 1$},\\
\rme_Q^0(U) = \ell & \mbox{if $i = 2$},\\
-\rme_Q^1(U) = -\binom{\ell}{2} & \mbox{if $i = 3$}.
\end{array}
\right.
$$
Therefore $ \g_s(Q;A) =\ell_A(A/Q)-\e_Q^0(A)+\e_Q^1(A)=1$ because $\ell_A(A/Q)=2$.

On the other hand, since  $A/\a$ is a Gorenstein ring and $\rmH_\fkm^1(A) \cong \rmH_\m^1(A/\a)$, we get 
$$\hdeg_Q(A_1) = \hdeg_Q(A/\a)=\e_Q^0(A/\a)=\e_{\m}^0(A/\a) = \ell.$$
We also have, for $i=0,2$, $\hdeg_Q(A_i)=0$ since $\H_{\m}^i(A)=0$.
Therefore
$$ \hdeg_Q(A)=\e_Q^0(A)+\sum_{j=0}^{2}\binom{2}{j}\hdeg_Q(M_j)=\e_Q^0(A)+2\hdeg_Q(A_1)=1+2\ell, $$
$$ \rmT_Q^1(A)=\sum_{j=1}^2\binom{2}{j-1}\hdeg_Q(A_j)=\hdeg_Q(A_1)=\ell, \ \ \mbox{and} \ \ \rmT_Q^2(A)=\hdeg_Q(A_1)=\ell$$
as required.
\end{proof}


\end{document}